\documentclass[10pt,a4paper]{amsart}
\usepackage[utf8]{inputenc}

\usepackage{graphics}
\usepackage{hyperref}
\usepackage[all]{xy}
\newcommand{\rightperp}[1]{#1^{\perp}}
\newcommand{\mathcolon}{\colon\,}
\newcommand{\leftperp}[1]{{}^\perp #1}
\DeclareSymbolFont{usualmathcal}{OMS}{cmsy}{m}{n}
\DeclareSymbolFontAlphabet{\mathcal}{usualmathcal}

\usepackage{url}
\usepackage{amsmath}
\usepackage[english]{babel}
\usepackage{amsthm}
\usepackage{mathrsfs}

\newtheorem{theorem}{Theorem}[section]

\newtheorem{lemma}[theorem]{Lemma}

\newtheorem{corollary}[theorem]{Corollary}

\theoremstyle{definition}
\newtheorem{definition}[theorem]{Definition}
\newtheorem{remark}[theorem]{Remark}

\newtheorem{ipg}[theorem]{}

\def\Qcoh{\mathfrak{Qcoh}}

\def\VF{\mathcal{V}\!\mathcal{F}}
\newcommand{\Ftot}[1]{\mathbf{F}_{{\rm tac}}(#1)}
\newcommand{\Ftt}[1]{\mathbf{F}_{{\rm tac}}^{#1}}
\def\RFML{\textrm{rFlatML}}
\def\FML{\textrm{FlatML}}
\def\CA{\,\mathcal{C}\!\!\mathcal{A}}

\newcommand{\shcplx}[1]{\mathscr{#1}_\bullet}
\usepackage{amsfonts}
\usepackage{amssymb}

\newcommand{\pd}{\mathrm{pd}}

\newcommand{\bbN}{\mathbb{N}}
\newcommand{\bbZ}{\mathbb{Z}}

\newcommand{\Spec}{\operatorname{Spec}} 
\newcommand{\OO}{\mathcal{O}}              

\newcommand{\Hom}{\operatorname{Hom}}

\newcommand{\Ext}{\operatorname{Ext}}

\newcommand{\Ker}{\operatorname{Ker}}

\newcommand{\A}{\mathcal{A}}
\newcommand{\B}{\mathcal{B}}
\newcommand{\C}{\mathcal{C}}

\newcommand{\F}{\mathcal{F}}
\newcommand{\G}{\mathcal{G}}
\newcommand{\clH}{\mathcal{H}}

\newcommand{\M}{\mathcal{M}}

\newcommand{\U}{\mathcal{U}}

\newcommand{\W}{\mathcal{W}}

 \newcommand{\Ch}{\mathrm{Ch}}
\newcommand{\tilclass}[1]{\widetilde{#1}}

\newcommand{\Modl}[1]{{#1}\textrm{-}\mathrm{Mod}}

\newcommand{\Flat}{\mathrm{Flat}}
\newcommand{\Cot}{\mathrm{Cot}}

\newcommand{\Proj}{\mathrm{Proj}}

\newcommand{\Vect}{\mathrm{Vect}}

\newcommand{\Filt}[1]{\operatorname{Filt}{#1}}

\newcommand{\Ho}{\operatorname{Ho}}

\theoremstyle{plain}
\newtheorem{thm}{Theorem}[section]
\newtheorem{lem}[thm]{Lemma}
\newtheorem{prop}[thm]{Proposition}
\newtheorem{cor}[thm]{Corollary}

\theoremstyle{definition}
\newtheorem{defn}[thm]{Definition}

\theoremstyle{remark}

\begin{document}
\title[Quillen equivalent models for the derived category of flats]{Quillen equivalent models for the derived category of flats and the resolution property }

\author{Sergio Estrada}
\address{(S.E.) Departamento de Matem\'aticas, Universidad de Murcia, 30100 Murcia, Spain}
\email{sestrada@um.es}

\author{Alexander Sl\'avik}
\address{(A.S.) Charles University, Faculty of Mathematics and Physics, Department of Algebra, Sokolovsk\'a 83, 186 75 Prague 8, Czech Republic}
\email{Slavik.Alexander@seznam.cz}

\thanks{The first author is supported by the grant MTM2016-77445-P and FEDER funds and the grant 19880/GERM/15 by the Fundaci\'on S\'eneca-Agencia de Ciencia y Tecnolog\'{\i}a de la Regi\'on de Murcia}
\date{}



\begin{abstract}
We investigate under which assumptions a subclass of flat quasi-coherent shea\-ves on a quasi-compact and semi-separated scheme allows to ``mock'' the homotopy category of projective modules. Our methods are based on module theoretic properties of the subclass of flat modules involved as well as their behaviour with respect to Zariski localizations. As a consequence we get that, for such schemes, the derived category of flats is equivalent to the derived category of very flats. If, in addition, the scheme satisfies the resolution property then both derived categories are equivalent to the derived category of infinite-dimensional vector bundles. The equivalences are inferred from a Quillen equivalence between the corresponding models.
\end{abstract}

\maketitle

\section{Introduction}

Throughout the paper $R$ will denote a commutative ring. 
In \cite{Nee} Neeman gives a new description of the homotopy category $\mathbf{K}(\Proj(R))$ as a quotient of $\mathbf{K}(\Flat(R))$. The main advantage of the new description is that it does not involve projective objects, so it can be generalized to non-affine schemes (see \cite[Remark 3.4]{Nee2}). So, in his thesis \cite{Mur}, Murfet \emph{mocks} the homotopy category of projectives on a non-affine scheme, by considering the category ${\bf D}(\Flat(X))$\footnote{\,The original terminology in \cite{Mur} for ${\bf D}(\Flat(X))$ was $\mathbf{K}_m(\Proj (X))$. This is referred in \cite{MS} as the \emph{pure derived category of flat sheaves} on $X$ and denoted by ${\bf D}(\Flat(X))$.} defined as the Verdier quotient $${\bf D}(\Flat(X)):=\frac{{\bf K}(\Flat(X))}{\widetilde{\Flat\mkern 0mu}^{\scalebox{0.6}{\bf K}}\!\!(X)},$$ where $ \widetilde{\Flat\mkern 0mu}^{\scalebox{0.6}{\bf K}}(X)$ denotes the class of acyclic complexes in ${\bf K}(\Flat(X))$ with flat cycles. In the language of model categories, Gillespie showed in \cite{G} that ${\bf D}(\Flat(X))$ can be realized as the homotopy category of a Quillen model structure on the category $\Ch(\Qcoh(X))$ of unbounded chain complexes of quasi-coherent sheaves on a quasi com\-pact and semi-separated scheme, and that, in fact, in case $X=\Spec(R)$ is affine, both homotopy categories ${\bf D}(\Flat(X))$ and $\mathbf{K}(\Proj(R))$ are triangle equivalent, coming from a Quillen equivalence between the corresponding models.

However, from an homological point of view, flat modules are much more complicated than projective modules. For instance, for a general commutative ring, the exact category of flat modules has infinite homological dimension. In order to partially remedy these complications, recently Positselski in \cite{P} has introduced a refinement of the class of flat quasi-coherent sheaves, the so-called \emph{very flat} quasi-coherent sheaves (see Section \ref{section.veryflat} for the definition and main properties) and showed that this class shares many nice properties with the class of flat sheaves, but it has potentially several advantages with respect to it, for instance, it can be applied to matrix factorizations (see the introduction of the recent preprint \cite{PS} for a nice and detailed treatment on the goodness of the very flat sheaves).

Moreover, in the affine case $X=\Spec(R)$, the exact category of very flat modules has finite homological dimension (every very flat module has projective dimension $\leq 1$). Therefore one easily obtains in this case a triangulated equivalence between ${\bf D}(\VF(R))$ and $\mathbf{K}(\Proj(R))$ (here $\VF(R)$ denotes the class of very flat $R$-modules). In particular it is much less involved than the aforementioned triangulated equivalence between ${\bf D}(\Flat(R))$ and $\mathbf{K}(\Proj(R))$ (\cite[Theorem 1.2]{Nee}).

So, if we denote by $\VF(X)$ the class of very flat quasi-coherent sheaves, one can also think in ``mocking'' the homotopy category of projectives over a non-affine scheme by defining the Verdier quotient $${\bf D}(\VF(X)):=\frac{{\bf K}(\VF(X))}{\widetilde{\VF\mkern 0mu}^{\scalebox{0.6}{\bf K}}\!\!(X)}.$$ It is then natural to wonder whether or not the (indirect) triangulated equivalence between ${\bf D}(\Flat(R))$ and ${\bf D}(\VF(R))$ still holds over a non-affine scheme. This was already proved to be the case for a semi-separated Noetherian scheme of finite Krull dimension in \cite[Corollary 5.4.3]{P}. As a first consequence of the results in this paper, we extend in Corollary \ref{cor.triang.equiv.flatveryflat} this result for arbitrary (quasi-compact and semi-separated) schemes.

\medskip\par\noindent
{\bf Corollary 1.} For any scheme $X$, the categories  ${\bf D}(\Flat(X))$ and ${\bf D}(\VF(X))$ are triangle equivalent.

\medskip\par Recall from Totaro \cite{Totaro} (see Gross \cite{Gross} for the general notion) that a scheme $X$ satisfies the \emph{resolution property} provided that $X$ has enough locally frees, that is, for every quasi-coherent sheaf $\mathscr{M}$ there exists an exact map $\oplus_i \mathscr{V}_i\to \mathscr{M}\to 0$, for some family $\{\mathscr{V}_i: i\in I\}$ of vector bundles. In this case the class of infinite-dimensional vector bundles (in the sense of Drinfeld \cite{D}) constitutes the natural extension of the class of projective modules for non-affine schemes. And one can define the derived category of infinite-dimensional vector bundles again as the Verdier quotient  $${\bf D}(\Vect(X)):=\frac{{\bf K}(\Vect(X))}{\widetilde{\Vect\mkern 0mu}^{\scalebox{0.6}{\bf K}}\!\!(X)}.$$ This definition trivially agrees with $\mathbf{K}(\Proj(R))$ in case $X=\Spec(R)$ is affine. By using the class of very flat sheaves we obtain in Corollary \ref{eq.vbundles} the following meaningful consequence, which does not seem clearly to admit a direct proof (i.e. a proof without using very flat sheaves).

\medskip\par\noindent
{\bf Corollary 2.} Let $X$ be a quasi-compact and semi-separated scheme satisfying the resolution property (for instance if $X$ is divisorial \cite[Proposition 6(a)]{Mur2}). Murfet's and Neeman's derived category of flats, ${\bf D}(\Flat(X))$, is triangle equivalent to ${\bf D}(\Vect(X))$, the derived category of infinite-dimensional vector bundles.

\medskip\par Indeed the methods developed in this paper go beyond the class of very flat quasi-coherent sheaves. More precisely, we investigate which are the conditions that a subclass  $\A_{\rm qc}$ of flat quasi-coherent sheaves has to fulfil in order to get a triangle equivalent category to ${\bf D}(\Flat(X))$. In fact, we show that the triangulated equivalence comes from a Quillen equivalence between the corresponding models. We point out that there are well-known examples of non Quillen equivalent models with equivalent homotopy categories. The precise statement of our main result is in Theorem \ref{t.mc.general case} (see the setup in Section \ref{section.q_equivalent} for unexplained terminology).

\medskip\par\noindent
{\bf Theorem.} Let $X$ be a quasi-compact and semi-separated scheme and let $\mathcal P$ be a property of modules and $\A$ its associated class of modules. Assume that $\A\subseteq \Flat$, and that the following conditions hold:
\begin{enumerate}
\item The class $\A$ is Zariski-local.
\item For each $R=\OO_X(U)$, $U\in \U$, the pair $(\A_R,\B_R)$ is a hereditary cotorsion pair generated by a set.
\item For each $R=\OO_X(U)$, $U\in \U$, every flat $\A_R$-periodic module is trivial. 
\item $j_*(\A_{{\rm qc}(U_{\alpha})})\subseteq \A_{{\rm qc}(X)}$, for each $\alpha\subseteq \{0,\ldots,m\}$.
\end{enumerate}
Then the class $\A_{{\rm qc}}$ defines an abelian model category structure in $\Ch(\Qcoh(X))$ whose homotopy category $\mathbf D\mathbb(\A_{{\rm qc}})$ is triangle equivalent to  ${\bf D}(\mathrm{Flat}(X))$, induced by a Quillen equivalence between the corresponding model categories. 

\medskip\par It is interesting to observe that conditions (1), (2) and (3) in the previous theorem only involve properties of modules. Thus we find useful and of independent interest to explicitly state in Theorem \ref{t.mc.affine case} the affine version of the previous theorem (and give an easy proof). Section \ref{section.examples} is meant to make abundantly clear the variety of examples of classes of modules that fit into those conditions. Of particular interest is the class $\mathcal A(\kappa)$ of \emph{restricted flat Mittag-Leffler modules} considered in Theorem \ref{rest.fml} which has been widely studied in the literature in the recent years (see, for instance, \cite{EGPT,EGT,GT,HT,Sar}). So regarding this class, we obtain the following meaningful consequences:

\medskip\par\noindent
{\bf Corollary.} Let $\kappa$ be an infinite cardinal and $\A(\kappa)$ be the class of $\kappa$-restricted flat Mittag-Leffler modules (notice that $\A(\kappa)=\Proj(R)$ in case $\kappa=\aleph_0$).
\begin{enumerate}
\item Every pure acyclic complex with components in $\A(\kappa)$ has cycles in $\A(\kappa)$.
\item The categories ${\bf D}(\A(\kappa))$ and ${\bf K}(\Proj(R))$ are triangle equivalent.
\end{enumerate}
\medskip\par The proof of (1) can be found in Theorem \ref{rest.fml} whereas the proof of (2) is a particular instance of Theorem \ref{t.mc.affine case} with $\A=\A(\kappa)$. In the special case $\kappa=\aleph_0$, the statement (1) recovers a well-known result due to Benson and Goodearl  (\cite[Theorem 1.1]{BG}).

\section{Preliminaries}

\begin{ipg}\underline{Zariski local classes of modules}. Let $\mathcal{P}$ be a property of modules and let $\A$ be the corresponding class of modules satisfying $\mathcal P$, i.e. for any ring $R$, the class $\A_R$ consists of $M\in \Modl R$ such that $M$ satisfies $\mathcal P_R$. 
We define the class $\A_{{\rm qc}(X)}$ in $\Qcoh(X)$ (or just $\A_{\rm qc}$ if the scheme is understood) as the class of all quasi-coherent sheaves $\mathscr M$ such that, for each open affine $U\subseteq X$, the module of sections $\mathscr M(U)\in \A_{\OO_X(U)}$.
We will be only interested in those properties of modules $\mathcal P$ such that the property of being in $\A_{{\rm qc}(X)}$ can be tested on an open affine covering of $X$. In this case we will say that the class $\A$ of modules (associated to $\mathcal P$) is \emph{Zariski-local}.\\
The following is a specialization of the \emph{ascent-descent} conditions (\cite[Definition 3.4]{EGT}) that suffices to prove Zariski locality (see Vakil \cite[Lemma 5.3.2]{Vakil} and also \cite[\S 27.4]{SP}):

\begin{lem}\label{ZL}
The class of modules $\A$ associated to the property of modules $\mathcal P$ is Zariski-local if and only if satisfies the following:
\begin{enumerate}
\item If an $R$-module $M\in \A_R$, then $M_f\in \A_{R_f}$ for all $f\in R$.
\item If $\left (f_1,\ldots,f_n\right )=R$, and $M_{f_i}=R_{f_i}\otimes_R M\in \A_{R_{f_i}}$, for all $i\in \{1,\ldots,n\}$, then $M\in \A_R$.
\end{enumerate}
\end{lem}
It is easy to see that the class $\Flat$ of flat modules is Zariski-local. A module $M$ is \emph{Mittag-Leffler} provided that the canonical map $M\otimes_R\prod_{i\in I}M_i\to \prod_{i\in I}M\otimes_R M_i$ is monic for each family of left $R$-modules $(M_i|\, i\in I)$. The classes $\FML$ (of flat Mittag-Leffler modules) and $\Proj$ (of projective modules) are also Zariski-local by 3.1.4.(3) and 2.5.2 in \cite[Seconde partie, 2.5.2]{RG}. The class $\RFML$ of \emph{restricted} flat Mittag-Leffler modules (in the sense of \cite[Example 2.1(3)]{EGT}) is also Zariski-local by \cite[Theorem 4.2]{EGT}.
\end{ipg}

\begin{ipg}\underline{Precovers, envelopes and complete cotorsion pairs}.\label{cotorsion-pairs} 
Throughout this section the symbol $\G$ will denote an abelian category. Let $\C$ be a class of objects in $\G$. A morphism
  $C\stackrel{\phi}{\rightarrow}{M}$ in $\G$ is called a
  \emph{$\C$-precover} if $C$ is in $\C$
  and $\Hom_{\G}(C',C) \to \Hom_{\G} (C',M) \to 0$ is exact
  for every $C' \in \C$. 
  If every object in $\G$ has a $\C$-precover, then the class
  $\C$ is called \emph{precovering}.
The dual notions are \emph{preenvelope} and \emph{preenveloping} class.

A pair $(\A,\B)$ of classes of objects in $\G$ is a \emph{cotorsion pair} if $\rightperp{\A}=\B$ and $\A = \leftperp{\B}$, where, given a class $\C$ of objects in $\A$, the right orthogonal  $\rightperp{\C}$ is defined to be the class of all $Y \in \G$ such that $\Ext^1_{\G}(C,Y) = 0$ for all $C \in \C$. The left orthogonal $\leftperp{\C}$ is defined similarly. A cotorsion pair $(\A,\B)$ is called \emph{hereditary} if $\Ext^i_{\G}(A,B) = 0$ for all $A \in \A$, $B \in \B$, and $i \geqslant 1$. A cotorsion pair $(\A,\B)$ is \emph{complete} if it has \emph{enough projectives} and \emph{enough injectives}, i.e.~for each $D \in \G$ there exist short exact sequences $0 \xrightarrow{} B \xrightarrow{} A \xrightarrow{} D \xrightarrow{} 0$ (enough projectives)
and $0 \xrightarrow{} D \xrightarrow{} B' \xrightarrow{} A' \xrightarrow{} 0$ (enough injectives)  with $A,A' \in \A$ and $B,B' \in \B$. It is then easy to observe that $A\xrightarrow{} D$ is an $\A$-precover of $D$ (such precovers are called \emph{special}). Analogously, $D\xrightarrow{} B'$ is a \emph{special} $\B$-preenvelope of $D$. A cotorsion pair $(\A,\B)$ is \emph{generated by a set} provided that there exists a \emph{set} $\mathcal{S}\subseteq \A$ such that $\rightperp{\mathcal{S}}=\B$. In case $\G$ is, in addition Grothendieck, it is known that a cotorsion pair generated by a set $\mathcal{S}$ which contains a generating set of $\G$ is automatically complete.

\end{ipg}
\begin{ipg}\underline{Exact model categories and Hovey triples}. In \cite{hovey} Hovey relates complete cotorsion pairs with abelian (or exact) model category structures.

An \emph{abelian model structure} on $\G$, that is, a model structure on $\G$ which is compatible with the abelian structure in the sense of \cite[Definition~2.1]{hovey}, corresponds by \cite[Theorem 2.2]{hovey} to a triple $(\C,\W,\F)$ of classes of objects in $\A$ for which $\W$ is thick\footnote{\,Recall that a class $\W$ in an abelian (or, more generally, in an exact) category $\G$ is \emph{thick} if it is closed under direct summands and satisfies that whenever two out of three of the terms in a short exact sequence are in $\W$, then so is the third.} and $(\C \cap \W,\F)$ and $(\C,\W \cap \F)$ are complete cotorsion pairs in $\G$.  In the model structure on $\G$ determined by such a triple, $\C$ is precisely the class of cofibrant objects, $\F$ is precisely the class of fibrant objects, and $\W$ is precisely the class of trivial objects (that is, objects weakly equivalent to zero). Such triple is often referred as a \emph{Hovey triple}.

Gillespie extends in \cite[Theorem~3.3]{G4} Hovey's correspondance, mentioned above, from the realm of abelian categories to the realm of weakly idempotent complete exact categories (\cite[Definition~2.2]{G4}). More precisely, if $\G$ is a  weakly idempotent complete exact categories (not necessarily abelian), then an \emph{exact model structure} on $\G$ (i.e. a model structure on $\G$ which is compatible with the exact structure in the sense of \cite[Definition~3.1]{G4}) corresponds to a Hovey triple $(\C,\W,\F)$ in $\G$.

\end{ipg}
\begin{ipg}\underline{Deconstructible classes}.\label{deconst}
A well ordered direct system,
  $(M_{\alpha}:\, \alpha\le \lambda)$, of objects in $\G$ is
  called \emph{continuous} if $M_0=0$ and, for each limit
  ordinal $\beta\leq \lambda$, we have $M_{\beta} =
  \varinjlim_{\alpha<\beta } M_{\alpha}$. If all morphisms in the
  system are monomorphisms, then the system is called a
  \emph{continuous directed union}.

  Let $\mathcal{S}$ be a class of objects in $\G$. An object $M$ in
  $\G$ is called \emph{$\mathcal{S}$-filtered} if there is a continuous
  directed union $(M_{\alpha}:\,  \alpha\le \lambda)$ of
  subobjects of $M$ such that $M = M_{\lambda}$ and for every
  $\alpha<\lambda$ the quotient $M_{\alpha+1}/M_{\alpha}$ is
  isomorphic to an object in $\mathcal{S}$. We denote by $\Filt(\mathcal{S})$ the
  class of all $\mathcal{S}$-filtered objects in $\G$. A class $\C$ is called \emph{deconstructible}
  provided that there exists a set $\mathcal{S}$ such that $\C=\Filt(\mathcal{S})$ (see \cite[Definition 1.4]{Sto}).
  It is then known by \cite[Theorem pg.195]{Sto} that any deconstructible class is precovering.

\end{ipg}
\begin{ipg}\underline{Chain complexes of modules}.
 We denote by $\Ch(\G)$ the category
of unbounded chain complexes of objects in $\G$, i.e. complexes $G_\bullet$ of
the form
$$\cdots\to G_{n+1}\xrightarrow{d^G_{n+1}}G_n\xrightarrow{d^G_{n}}
G_{n-1}\to\cdots.$$
We will denote by $Z_n G_\bullet$ the \emph{$n$-cycle of} $G$,
i.e. $Z_nG=\Ker(d^G_n)$. Given a chain complex $G$ the
\emph{$n^{th}$-suspension of $G$}, $\Sigma^n G$, is the complex
defined as $(\Sigma^n G)_k=G_{k-n}$ and
$d^{\Sigma^n G}_k=(-1)^n d_{k-n}$. And for a given object $A\in \G$,
the \emph{$n$-disk} complex $D^n(A)$ is the complex with the object $A$ in the
components $n$ and $n-1$, $d_n$ as the identity map, and 0
elsewhere.

We denote by ${\bf K}(\G)$ the homotopy category of $\mathcal G$, i.e. ${\bf K}(\G)$ has the same objects as $\Ch(\G)$ and the morphisms are the homotopy classes of morphisms of chain complexes.

In case $\G=\Modl R$, we will denote $\Ch(\G)$ (resp. ${\bf K}(\G)$) simply by $\Ch(R)$ (resp. ${\bf K}(R)$).
  Given a class $\C$ in $\G$, we shall consider the following classes of chain complexes:

\begin{itemize}
\item $\Ch(\C)$ (resp. ${\bf K}(\C)$)  is the full
  subcategory of $\Ch(\G)$ (resp. of ${\bf K}(\C)$)  of all complexes $C_\bullet\in\Ch(\G)$ such that
  $C_n\in\C$. 
\item $\Ch_{\textrm{ac}}(\C)$ (resp. ${\bf K}_{\textrm{ac}}(\C)$) is the class of all acyclic complexes in $\Ch(\C)$ (resp. in ${\bf K}(\C)$).
\item $\widetilde{\C}$ (resp. $\widetilde{\C\mkern 0mu}^{\scalebox{0.6}{\bf K}}$) is the class class of all complexes
  $C_\bullet\in \Ch_{\textrm{ac}}(\C)$ (resp. $C_\bullet\in {\bf K}_{\textrm{ac}}(\C)$) with the cycles $Z_nC_\bullet$ in $\C$ for all $n\in \bbZ$. A complex in $\widetilde{\C}$ is called a \emph{$\mathcal C$ complex}.

\item If $(\A, \B)$ is a cotorsion pair in $\G$, then $dg(\A)$ is
  the class of all complexes $A_\bullet\in \Ch(\A)$ such that every morphism
  $f\mathcolon A_\bullet\to B_\bullet$, with $B_\bullet$ a $\B$ complex, is null
  homotopic. Since $\Ext^1_{\G}(A_n, B_n)=0$ for every $n\in \bbZ$, a
  standard formula allows to infer that $dg(\A)={}^\perp{} \tilclass
  \B$.
  Analogously, $dg(\B)$ is the class of all complexes $B_\bullet \in \Ch(\B)$ such
  that every morphism $f\mathcolon A_\bullet\to B_\bullet$, with $A_\bullet$ an $\A$ complex, is null
  homotopic. Hence $dg(\B)=\widetilde{\A}{}^\perp{}$.
\end{itemize}
\end{ipg}

\section{Very flat modules and sheaves}\label{section.veryflat}
One of the main application of the results in this paper concerns the classes of very flat modules and very flat quasi-coherent sheaves, as defined by Positselski in \cite{P}. In the present section we summarize all relevant definitions and properties regarding this class and that will be relevant in the sequel.
\begin{ipg} 
\underline{Very flat and contraadjusted modules}. Let us consider the set $\mathcal S=\{R[r^{-1}]:\ r\in R\}$ and let  $(\VF(R),\CA(R))$ the complete cotorsion pair generated by $\mathcal S$. The modules in the class $\VF(R)$ are called \emph{very flat} and the modules in the class $\CA(R)$ are called \emph{contraadjusted}. It is then clear that every projective module is very flat, and that every very flat module is, in particular, flat. In fact it is easy to observe that every very flat module has finite projective dimension $\leq 1$. Thus, the complete cotorsion pair $(\VF,\CA)$ is automatically hereditary and $\CA$ is closed under quotients. We finally notice that  $L$ is very flat in any short exact sequence $0\to L\to V\to M\to 0$ in which $V$ is very flat and $\pd_R(M)\leq 1$ (where $\pd_R(M)$ is the projective dimension of $M$).
\end{ipg}
\begin{prop}[Positselski]
The class of very flat modules is Zariski-local.
\end{prop}
\begin{proof}
Condition $(1)$ of Lemma \ref{ZL} holds by \cite[Lemma 1.2.2(b)]{P}.\\
Condition $(2)$ of Lemma \ref{ZL} follows from \cite[Lemma 1.2.6(a)]{P}.
\end{proof}
\begin{ipg} \underline{Very flat and contraadjusted quasi-coherent sheaves}.
Let $X$ be any scheme. A quasi-coherent sheaf $\mathscr{M}$ is \emph{very flat} if there exists an open affine covering $\U$ of $X$ such that $\mathscr{M}(U)$ is a very flat $\mathcal O_X(U)$-module for each $U\in \U$. By the previous proposition, the definition of very flat quasi-coherent sheaf is independent of the choice of the open affine covering.  A quasi-coherent sheaf $\mathscr{N}$ is \emph{contraadjusted} if $\Ext^n(\mathscr{M},\mathscr{N})=0$ for each very flat quasi-coherent sheaf $\mathscr{M}$ and every integer $n\geq 1$. 

Since the class of very flat modules is resolving (i.e. closed under kernels of epimorphisms) we infer that the class of very flat quasi-coherent sheaves is also resolving.
\end{ipg}
\begin{ipg} 
\underline{Very flat generators in $\Qcoh(X)$}.
Let $X$ be a quasi-compact and semi-separated sche\-me, with $\U=\{U_0,\cdots, U_d\}$ a semi-separated finite affine covering of $X$.
Let $U=U_{i_0}\cap\cdots\cap~U_{i_p}$ be any intersection of open sets in the cover $\mathfrak U$ and let $j:U\hookrightarrow X$ be the inclusion of $U$ in $X$. The inverse image functor $j^*$ is just the restriction, so it is exact and preserves quasi-coherence. The direct image functor $j_*$ is exact and preserves quasi-coherence because $j:U\hookrightarrow X$ is an affine morphism, due to the semi-separated assumption. Thus we have an adjunction $(j^*,j_{*})$ with $j_{*}:\Qcoh(U)\to \Qcoh(X)$ and $j^*:\Qcoh(X)\to \Qcoh(U)$.

The proof of the next proposition is implicit in \cite[Proposition 1.1]{Leo} (see also Murfet \cite[Proposition 3.29]{Mur} for a very detailed treatment) by noticing that the direct image functor $j_{*}$ preserves not just flatness but in fact \emph{very} flatness (by \cite[Corollary 1.2.5(b)]{P}). The reader can find a short and direct proof in \cite[Lemma 4.1.1]{P}.
\begin{prop}\label{prop.vf.generators}
Let $X$ be a quasi-compact and semi-separated scheme. Every quasi-co\-he\-rent sheaf is a quotient of a very flat quasi-coherent sheaf. Therefore $\Qcoh(X)$ possesses  a family of very flat generators.
\end{prop}

\end{ipg}

\begin{ipg}\underline{The very flat cotorsion pair in $\Qcoh(X)$}. For any scheme $X$, the class $\VF(X)$ of very flat quasi-coherent sheaves is deconstructible (by \cite[Corollary 3.14]{EGPT}). Therefore the class of very flat quasi-coherent sheaves is a precovering class (see \ref{deconst}). If, in addition, the scheme $X$ is quasi-compact and semi-separated we infer from \cite[Corollary 3.15]{EGPT} and \cite[Corollary 4.1.2]{P} that the pair $(\VF(X),\CA(X))$ is a complete hereditary cotorsion pair in $\Qcoh(X)$ (where $\CA(X)$ denotes the class of all contraadjusted quasi-coherent sheaves on $X$).

By \cite[Lemma 1.2.2(d)]{P} the class of very flat modules (and hence the class of very flat quasi-coherent sheaves) is closed under tensor products. Thus, in case $X$ is quasi-compact and semi-separated, \cite[Theorem 4.5]{EGPT} yields a cofibrantly generated and monoidal model category structure in $\mathrm{Ch}(\Qcoh(X))$ where the weak equivalences are the homology isomorphisms. The cofibrations (resp. trivial cofibrations) are monomorphisms whose cokernels are dg-very flat complexes (resp. very flat complexes). The fibrations (resp. trivial fibrations) are epimorphisms whose kernels are dg-contraadjusted complexes (resp. contraadjusted complexes). Therefore the corresponding triple is $$(dg(\VF(X)),\Ch_{\rm ac}(\Qcoh(X)),dg(\CA(X))).$$
\end{ipg}
\section{The property of modules involved. Examples}\label{section.examples}
As we will see in the next sections, we are mainly concerned in deconstructible classes of modules that are closed under certain periodic modules. We start by recalling the notion of $\C$-periodic module with respect to a class $\C$ of modules.
\begin{defn}
Let $\mathcal C$ be a class of modules. A module $M$ is called $\mathcal C$-periodic if there exists a short exact sequence $0\to M\to C\to M\to 0$, with $C\in \mathcal C$.
\end{defn}
The following proposition relating flat periodic $\mathcal A$-modules and acyclic complexes with components in $\mathcal A$ is standard, but relevant for our purposes. The reader can find a proof in \cite[Proposition 1 and Proposition 2]{EFI}.
\begin{prop}\label{prop.periodic}
Let $\mathcal A$ be a class of modules closed under direct sums and direct summands. The following are equivalent:
\begin{enumerate}
\item Every cycle of an acyclic complex with flat cycles and with components in $\mathcal A$ belongs to $\mathcal A$.
\item Every flat $\mathcal A$-periodic module belongs to $\mathcal A$.
\end{enumerate}
\end{prop}
We are interested in deconstructible classes of modules $\A$ satisfying condition $(2)$ in the previous proposition. Of course the first trivial example is the class $\Flat(R)$ of flat modules itself. Since the class of all flat Mittag-Leffler modules is closed under pure submodules, this class also trivially yields an example of a class $\A$ satisfying that every flat $\mathcal A$-periodic module is in $\A$. However this class has an important drawback: it is only deconstructible in the trivial case of a perfect ring (see Herbera and Trlifaj \cite[Corollary 7.3]{HT}). This setback can be remedied by considering the \emph{restricted} flat Mittag-Leffler modules, in the sense of \cite[Example 2.1(3)]{EGT}, as we will show in Theorem \ref{rest.fml} below.

Now we will provide with other interesting non-trivial examples of such classes $\A$ satisfying condition (2) above, and that will be relevant in the applications of our main results in the next sections.  

The first example is the class $\A=\Proj(R)$ of projective $R$-modules and goes back to Benson and Goodearl \cite[Theorem 1.1]{BG}.
\begin{prop}\label{bg}
Let $\Proj(R)$ be the class of all projective $R$-modules.
Every flat $\Proj(R)$-periodic module is projective. As a consequence every pure acyclic complex of projectives is contractible (i.e. has projective cycles).
\end{prop}
The second application is the class $\A=\VF(R)$ of very flat modules (this is due to \v S\v t'ov\' \i\v cek, personal communication).
\begin{prop}
Every flat $\VF(R)$-periodic module is very flat. As a consequence every pure acyclic complex of very flat modules has very flat cycles.
\end{prop}
\begin{proof}
Let $0\to F\to G\to F\to 0$ be an exact sequence with $F$ flat and $G$ very flat. Let $0\to F_1\to P\to F\to 0$ be an exact sequence with $P$ projective; then $F_1$ is flat. An application of the horseshoe lemma gives the following commutative diagram
 \[\xymatrix{
&0 \ar[d]  & 0 \ar[d] & 0 \ar[d] \\
0 \ar[r] & F_1 \ar[d] \ar[r] &Q\ar[r] \ar[d] & F_1\ar[d]  \ar[r] & 0 \\
0 \ar[r] & P \ar[d] \ar[r] &P\oplus P\ar[r] \ar[d] & P\ar[d]  \ar[r] & 0  \\
0 \ar[r] & F \ar[d]\ar[r] &G\ar[r] \ar[d] & F\ar[d]  \ar[r] & 0 \\
&0 & 0 & 0.
}
\] 
where  $Q$ is projective, since $\pd_R(G)\leq 1$. Thus, by Proposition \ref{bg}, $F_1$ is projective and therefore $\pd_R(F)\leq 1$.
Let $C\in \CA(R)$. Then applying $\Hom_R(-,C)$ to the short exact sequence yields $0=\Ext^1_R(G, C)\to \Ext^1_R(F, C)\to \Ext^2_R(F, C)=0$, hence $F\in \VF(R)$. Finally, the consequence follows from Proposition \ref{prop.periodic}(1) (with $\A=\VF(R)$).
\end{proof}
The last example is the announced deconstructible class of \emph{restricted} flat Mittag-Leffler modules as defined in \cite[Example 2.1(3)]{EGT}. 
\begin{thm}\label{rest.fml}
let $\kappa$ be an infinite cardinal and $\A(\kappa)$ be the class of $\kappa$-restricted flat Mittag-Leffler modules.
Every flat $\A(\kappa)$-periodic module is in $\A(\kappa)$. As a consequence every pure acyclic complex with components in $\A(\kappa)$ has cycles in $\A(\kappa)$.
\end{thm}
\begin{proof}
The proof mostly follows the pattern outlined in \cite{BG}; the main difference is that instead of direct sum decomposition, we work with filtrations and Hill Lemma (cf.\ \cite[Theorem 7.10]{GT}). Given a short exact sequence
\begin{equation}\label{BG-ses}
0 \to F \to G \stackrel f\to F \to 0
\end{equation}
with $F$ flat and $G \in \A(\kappa)$, we fix a Hill family $\clH$ for $G$. The goal is to pick a filtration $(G_\alpha \mid \alpha \leq \sigma)$ from $\clH$ such that for each $\alpha < \sigma$, $f(G_\alpha) = F \cap G_\alpha$, $f(G_\alpha) \subseteq_* F$, and $G_{\alpha+1}/G_\alpha$ is $\leq \kappa$-presented flat Mittag-Leffler.

Once this is achieved, we obtain a filtration of the whole short exact sequence \eqref{BG-ses} by short exact sequences of the form
\[ 0 \to F_{\alpha+1}/F_\alpha \to G_{\alpha+1}/G_\alpha \to F_{\alpha+1}/F_\alpha \to 0 \]
(putting $F_{\alpha} = f(G_\alpha) = F \cap G_\alpha$).
Since the property of being flat Mittag-Leffler passes to pure submodules (cf.\ \cite[Corollary 3.20]{GT}), this would make $F_{\alpha+1}/F_\alpha$ an $\leq\kappa$-presented flat Mittag-Leffler module and hence imply $F \in \A(\kappa)$. (Note that by \cite[Lemma 2.7 (1)]{EGPT}, each $\leq \kappa$-generated flat Mittag-Leffler module is (even strongly) $\leq \kappa$-presented.)

Put $G_0 = 0$. For limit ordinals $\alpha$, it suffices to take unions of already constructed submodules $G_\beta$, $\beta<\alpha$; note that by property (H2) in Hill Lemma, $G_\alpha \in \clH$ then. Having constructed modules up to $G_\alpha$ (and assuming $G_\alpha \neq G$), we construct $G_{\alpha+1}$ as follows: We pass to the quotient short exact sequence
\[ 0 \to F/F_\alpha \to G/G_\alpha \stackrel{\overline f}\to F/F_\alpha \to 0, \]
which, by assumption, satisfy that $F/F_\alpha$ is flat and $G/G_\alpha \in \A(\kappa)$. Note that $F/F_\alpha$, being (identified with) a pure submodule of $G/G_\alpha$, is flat Mittag-Leffler. The Hill family $\clH$ gives rise to family $\clH'$ for $G/G_\alpha$, which consists of factors of modules from $\clH$ (containing $G_\alpha$) by $G_\alpha$.

Let us first show that any $\leq\kappa$-generated submodule $Y$ of $G/G_\alpha$ can be enlarged to $\leq\kappa$-generated $\overline G \in \clH'$ with the property that $\overline f(\overline G) \subseteq_* F/F_\alpha$ and $\overline G \cap F/F_\alpha$ is $\leq\kappa$-generated. To this end, we construct inductively a chain of submodules $\overline G_n \in \clH'$ with union $\overline G$ (utilizing property (H2)). Let $\overline G_0$ be an arbitrary $\leq\kappa$-generated module $\overline G_0 \in \clH'$ containing $Y$ (obtained via (H4)). Assuming we have constructed $\overline G_n$, we get $\overline G_{n+1}$ by taking these steps:
\begin{enumerate}
\item Enlarge $\overline f(\overline G_n)$ to a $\leq\kappa$-generated pure submodule $X_n$ of $F/F_\alpha$; this is possible by \cite[Lemma 2.7 (2)]{EGPT} once we notice that $F/F_\alpha$, being a pure submodule of $G/G_\alpha$, is flat Mittag-Leffler.
\item Take $\leq\kappa$-generated $\overline G_{n+1} \in \clH'$ such that $X \subseteq \overline f(G_n')$; this is again possible by property (H4) of the Hill Lemma.
\end{enumerate}
We have $\overline f(\overline G) = \bigcup_{n \in \bbN} \overline f(\overline G) = \bigcup_{n \in \bbN} X_n \subseteq_* F/F_\alpha$. This also shows that $\overline f(\overline G)$ is flat Mittag-Leffler, hence $\leq\kappa$-presented. The short exact sequence
\[ 0 \to \overline G \cap (F/F_\alpha) \to \overline G \to \overline f(\overline G) \to 0 \]
now shows that $\overline G \cap (F/F_\alpha)$ is indeed $\leq\kappa$-generated.

Now iterate the claim as follows: Start with arbitrary $\leq\kappa$-generated non-zero $Y_0 \subseteq G/G_\alpha$ and obtain $\overline G_0$ from the claim. Enlarge it to $\overline G_1 \in \clH'$ satisfying $\overline G_0 \cap (F/F_\alpha) \subseteq \overline f(\overline G_1)$ (which we may do using (H4), since $\overline G_0 \cap (F/F_\alpha)$ is $\leq\kappa$-generated). Taking $Y_1 = \overline G_1 + \overline f(\overline G_1)$ and applying the claim, we get $\overline G_2$ etc. This way we obtain a chain
\[ \overline G_0 \cap (F/F_\alpha) \subseteq \overline f(\overline G_1) \subseteq \overline G_2 \cap (F/F_\alpha) \subseteq \overline f(\overline G_3) \subseteq \ldots, \]
so for $\overline G = \bigcup_{n \in \bbN} \overline G_n \in \clH'$ we have $\overline G \cap (F/F_\alpha) = \overline f(\overline G)$. Also the purity of $\overline f(\overline G)$ in $F/F_\alpha$ and being $\leq\kappa$-generated is ensured.

The desired module $G_{\alpha+1}$ is now the one satisfying $G_{\alpha+1}/G_\alpha = \overline G$.
\end{proof}
Note that in the case $\kappa = \aleph_0$, $\A(\kappa)$ is just the class of projective modules by \cite[Seconde partie, Section 2.2]{RG}, so this also covers the case of \cite{BG}.

\section{Quillen equivalent models for ${\bf K}(\Proj(R))$}

It is known (see Bravo, Gillespie and Hovey \cite[Corollary 6.4]{BGH}) that the homotopy category of projectives ${\bf K}(\Proj(R))$ can be realized as the homotopy category of the model $\M_{\rm proj}=(\Ch(\Proj(R)), \Ch(\Proj(R))^{\perp}, \Ch(R))$ in $\Ch(R)$. Now, by \cite[Remark 4.2]{G}, the class $\Ch(\Flat(R))$ induces model category in $\Ch(R)$ given by the triple, $$(\Ch(\Flat(R)),\Ch(\Proj(R))^{\perp}, dg(\Cot(R))).$$The last model is thus Quillen equivalent to $\M_{\rm proj}$. Therefore, its homotopy category, the derived category of flats ${\bf D}(\Flat(R))$, is triangulated equivalent to ${\bf K}(\Proj(R))$. The next theorem gives sufficient conditions on a class of modules $\A$ to get ${\bf D}(\A)$ and ${\bf D}(\Flat(R))$ to be triangulated equivalent. For concrete examples of such classes the reader should have in mind the classes of modules considered in Section \ref{section.examples}.
\begin{thm}\label{t.mc.affine case}
Let $\A\subseteq \Flat(R)$ be a class of modules such that:
\begin{enumerate}
\item The pair $(\A,\B)$ is a hereditary cotorsion pair generated by a set.
\item Every flat $\A$-periodic module is trivial.
\end{enumerate}
Then there is an abelian model category structure  $\mathcal M=(\mathrm{Ch}(\A),\Ch(\Proj(R))^{\perp},dg\widetilde{\B})$ in $\Ch(R)$. If we denote by $\mathbf D\mathbb(\A)$ the homotopy category
of $\M$, then   ${\bf D}(\mathrm{Flat}(R))$, $\mathbf D\mathbb(\A)$ and ${\bf K}(\Proj(R))$ are triangulated equivalent, induced by a Quillen equivalence between the corresponding model categories.
\end{thm}
\begin{proof}
Let $\mathcal M=(\mathrm{Ch}(\A),\W,dg\widetilde{\B})$ be the model associated to the complete hereditary cotorsion pairs $(\mathrm{Ch}(\A), \mathrm{Ch}(\A)^\perp)$ and $(\tilclass{\A},dg\tilclass{B})$ in $\Ch(R)$. To get the claim it suffices to show that $\W=\Ch(\Proj(R))^\perp$. To this aim we will use \cite[Lemma 4.3(1)]{G}, i.e. we need to prove:
\begin{enumerate}
\item[(i) ] $\tilclass{\A}=\Ch(\A)\cap \Ch(\Proj(R))^\perp$.
\item[(ii) ] $ \mathrm{Ch}(\A)^\perp\subseteq \mathrm{Ch}(\Proj(R))^\perp$.
\end{enumerate}
Condition (ii) is clear because $\Proj(R)\subseteq \A$. Now, by Neeman \cite[Theorem 8.6]{Nee}, $\Ch(\A)\cap \Ch(\Proj(R))^\perp=\tilclass{\Flat}(R)\cap \Ch(\A)$. But, by the assumption (2), we follow that  $\tilclass{\Flat}(R)\cap \Ch(\A)=\tilclass{\A}$.
\end{proof}
\begin{remark}
Starting with a class $\A$ in the assumptions of Theorem \ref{t.mc.affine case}, we may construct, for each integer $n\geq 0$, the class $\mathcal{A}^{\leq n}$ of modules  $M$ possessing an exact sequence $$0\to A_n\to A_{n-1}\to \ldots \to A_0\to M\to  0$$ with $A_i\in \A$, $i=1,\ldots, n$. The derived categories $\mathbf{D}(\mathcal A^{\leq n})$ and $\mathbf{D}(\mathcal A)$ are triangulated equivalent (see Positselski \cite[Proposition A.5.6]{P}). In particular we can infer from this a triangulated equivalence between ${\bf K}(\Proj(R))$ and ${\bf D}(\VF(R))$. By using a standard argument of totalization one can also check that $\mathbf{D}(\mathcal A^{\leq n})$ and $\mathbf{D}(\mathcal A)$ can be realized as the homotopy categories of two models $\M_1$ and $\M_2$ and that these models are Quillen equivalent without using Neeman \cite[Theorem 8.6]{Nee}. From this point of view it seems that the triangulated equivalence between  ${\bf K}(\Proj(R))$ and ${\bf D}(\VF(R))$ is much less involved than the one between ${\bf K}(\Proj(R))$ and ${\bf D}(\Flat(R))$.
\end{remark}

\section{Quillen equivalent models for ${\bf D}(\Flat(X))$}\label{section.q_equivalent}

\noindent
{\bf Setup:} Throughout this section $X$ will denote a quasi-compact and semi-separated scheme. If $\U=\{U_0,\ldots,U_m\}$ is an affine open cover of $X$ and $\alpha=\{i_0,\ldots, i_k\}$ is a finite sequence of indices in the set $\{0,\ldots,m\}$ (with $i_0<\cdots<i_k$), we write $U_{\alpha}=U_{i_0}\cap \cdots \cap U_{i_k}$ for the corresponding affine intersection.
 
\medskip 
In \cite{Mur} Murfet shows that the derived category of flat quasi-coherent sheaves on $X$, ${\bf D}(\Flat(X))$, constitutes a good replacement of the homotopy category of projectives for non-affine schemes, because in case $X=\Spec(R)$ is affine, the categories ${\bf D}(\Flat(X))$ and ${\bf K}(\Proj(R))$ are triangulated equivalent.  There is a model for ${\bf D}(\Flat(X))$ in $\Ch(\Qcoh(X))$ given by the triple $$\M_{\rm flat}=(\Ch(\Flat(X),\W, dg(\Cot(X))).$$(see \cite[Corollary 4.1]{G}). We devote this section to provide a general method to produce model categories $\M$ in $\Ch(\Qcoh(X))$ which are Quillen equivalent to $\M_{\rm flat}$. In particular this implies that the homotopy category $\Ho(\M)$ and ${\bf D}(\Flat(X))$ are triangulated equivalent.

\begin{theorem}\label{t.mc.general case}
Let $X$ be a scheme and let $\mathcal P$ be a property of modules and $\A$ its associated class of modules. Assume that $\A\subseteq \Flat$, and that the following conditions hold:
\begin{enumerate}
\item The class $\A$ is Zariski-local.
\item For each $R=\OO_X(U)$, $U\in \U$, the pair $(\A_R,\B_R)$ is a hereditary cotorsion pair generated by a set.
\item For each $R=\OO_X(U)$, $U\in \U$, every flat $\A_R$-periodic module is trivial. 
\item $j_*(\A_{{\rm qc}(U_{\alpha})})\subseteq \A_{{\rm qc}(X)}$, for each $\alpha\subseteq \{0,\ldots,m\}$.
\end{enumerate}
Then there is an abelian model category structure  $ \M_{\A_{{\rm qc}}}$ in $\Ch(\Qcoh(X))$ given by the triple $(\mathrm{Ch}(\A_{{\rm qc}}),\W,dg(\B))$.  If we denote by $\mathbf D\mathbb(\A_{{\rm qc}})$ the homotopy category
of $\M_{\A_{{\rm qc}}}$, then the categories   ${\bf D}(\mathrm{Flat}(X))$ and $\mathbf D\mathbb(\A_{{\rm qc}})$ are triangulated equivalent, induced by a Quillen equivalence between the corresponding model categories. In case $X=\Spec(R)$ is affine, $\mathbf D\mathbb(\A_{R})$ is triangulated equivalent to $\mathbf{K}(\Proj(R))$.
\end{theorem}

Before proving the theorem, let us focus on one particular instance of it: if we take $\A=\VF$ (the class of very flat modules) the theorem gives us that ${\bf D}(\mathrm{Flat}(X))$ and $\mathbf D\mathbb(\VF(X))$ are triangulated equivalent. This generalizes to arbitrary schemes \cite[Corollary 5.4.3]{P}, where such a triangulated equivalence is obtained for a semi-separated Noetherian scheme of finite Krull dimension.
\begin{cor}\label{cor.triang.equiv.flatveryflat}
For any scheme $X$, the categories ${\bf D}(\mathrm{Flat}(X))$ and $\mathbf D\mathbb(\VF(X))$ are triangulated equivalent.
\end{cor}

Let us prove Theorem \ref{t.mc.general case}. We firstly require the following useful lemma.

\begin{lemma}\label{l.keylemma}
Suppose $\A$ is as in Theorem \ref{t.mc.general case} (possibly without satisfying condition (3)). Then for any $\mathscr M_\bullet\in \Ch(\Flat(X))$ there exists a short exact sequence
$$0\to \mathscr K
_\bullet\to \mathscr F_\bullet\to \mathscr M_\bullet\to 0,$$
where $\mathscr F_\bullet\in \Ch(\A_{{\rm qc}(X)})$ and $\mathscr K_\bullet\in\tilclass\Flat(X)$. 
\end{lemma}
\begin{proof}
We essentially follow the proof of \cite[Lemma 4.1.1]{P1}; the main difference is that instead of sheaves, we are dealing with complexes of sheaves. Starting with the empty set, we gradually construct such a short exact sequence with the desired properties manifesting on larger and larger unions of sets from $\U$, reaching $X$ in a finite number of steps.

Assume that for an open subscheme $T$ of $X$ we have constructed a short exact sequence $0 \to \shcplx L \to \shcplx G \to \shcplx M \to 0$ such that the restriction $h^*(\shcplx G)$ belongs to $\Ch(\A_{{\rm qc}(T)})$ ($h\colon T \hookrightarrow X$ being the inclusion map) and $\shcplx L \in \tilclass\Flat(X)$. Let $U \in \U$ (with inclusion map $j\colon U \hookrightarrow X$); our goal is to construct a short exact sequence $0 \to \shcplx L' \to \shcplx G' \to \shcplx M \to 0$ with the same property with respect to the set $U \cup T$. Let us note that the adjoint pairs of functors on sheaves $(j^*, j_*)$, $(h^*, h_*)$ yield corresponding adjoint pairs of functors on complexes of sheaves.

Pick a short exact sequence $0 \to \shcplx K' \to \shcplx Z \to j^*(\shcplx G) \to 0$ of complexes of sheaves over the affine subscheme $U$, where $\shcplx Z \in \Ch(\A_{{\rm qc}(U)}) = \Ch(\A_{\OO_U(U)})$ and $\shcplx K' \in \Ch(\A_{\OO_U(U)})^{\perp}$, i.e.\ special precover in the category of complexes of $\OO_U(U)$-modules. In this (affine) setting we know from \cite{Nee} that $\shcplx K' \in \tilclass\Flat(U)$, since $\shcplx K' \in \Ch(\Flat(U)) \cap \Ch(\A_{\OO_U(U)})^{\perp} \subseteq \Ch(\Flat(U)) \cap \Ch(\Proj(U))^{\perp}$. Using the direct image functor, we get $0 \to j_*(\shcplx K') \to j_*(\shcplx Z') \to j_*j^*(\shcplx G) \to 0$ over $X$. Since $U \in \U$ is affine, $j_*$ is an exact functor taking flats to flats and also preserving $\A$ by condition (4), so $j_*(\shcplx K') \in \tilclass\Flat(X)$, whence $j_*(\shcplx Z)$ stays in $\Ch(\A_{{\rm qc}})$. Now considering the pull-back with respect to the adjunction morphism $\shcplx G \to j_*j^*(\shcplx G)$, one gets a new short exact sequence ending in $\shcplx G$; let $\shcplx G'$ be its middle term:
\[\xymatrix{
0 \ar[r] & j_*(\shcplx K') \ar[r]\ar@{=}[d] & \shcplx G' \ar[r]\ar[d] & \shcplx G \ar[r]\ar[d] & 0 \\
0 \ar[r] & j_*(\shcplx K') \ar[r] & j_*(\shcplx Z') \ar[r] & j_*j^*(\shcplx G) \ar[r] & 0
}\] 

Let us now check that the components of $\shcplx G'$ are in $\Ch(\A_{{\rm qc}(U \cup T)})$; this is sufficient to check on $U$ and $T$ separately. Firstly, $j^*(\shcplx G') \cong \shcplx Z$, which is in $\Ch(\A_{{\rm qc}(U)})$ by construction. On the other hand, the complex $j^*(\shcplx G)$, when further restricted to $U \cap T$, is in $\Ch(\A_{{\rm qc}(U \cup T)})$ ($\A$ being Zariski-local class), and the same holds for the complex $\shcplx K'$ by the resolving property of $\A$. The embedding $U \cap T \hookrightarrow T$ is an affine morphism (by semi-separatedness) and preserving $\A$ by (4), so $j_*(\shcplx K') \in \Ch(A_{{\rm qc}(T)})$. Therefore $\shcplx G'$, as an extension of $j_*(\shcplx K')$ by $\shcplx G$, belongs to $\Ch(A_{{\rm qc}(T)})$, too.

Finally, the kernel $\shcplx K$ of the composition of morphisms $\shcplx G' \to \shcplx G \to \shcplx M$ is an extension of $\shcplx L$ and $j_*(\shcplx K')$, hence a complex from $\tilclass\Flat(X)$. This proves the existence of the short exact sequence from the statement.
\end{proof}

\begin{proof}[Proof of Theorem \ref{t.mc.general case}]
First of all we notice that the class $\A_{{\rm qc}}$ contains a family of generators for $\Qcoh(X)$; this is just a variation of the idea used in the proof of \cite[Lemma 4.1.1]{P}, where we replace the class of very flats by $\A$ (and do not care about the kernel of the morphisms), which is possible thanks to property (4).

Then, by \cite[Corollary 3.15]{EGPT} we get in $\Qcoh(X)$ the complete hereditary cotorsion pair $(\A_{{\rm qc}},\B)$ generated by a set. Thus by \cite[Theorem 4.10]{G} we get the abelian model structure $\M_{\A_{{\rm qc}}}^{\rm qc}=(\Ch(\A_{{\rm qc}}),{\W_1}, dg(\B))$ in $\Ch(\Qcoh(X)) $ given by the two complete hereditary cotorsion pairs:
$$(\Ch(\A_{{\rm qc}}),\Ch(\A_{{\rm qc}})^{\perp})\quad \text{and}\quad (\tilclass{\A}_{{\rm qc}}, dg(\B)).$$
Since $\A_{{\rm qc}}\subseteq \Flat(X) $, we get the corresponding induced cotorsion pairs in $\Ch(\Flat(X))$ (with the induced exact structure from $\Flat(X)$):
$$(\Ch(\A_{{\rm qc}}),\Ch(\A_{{\rm qc}})^{\perp}\cap \Ch(\Flat(X)))\quad \text{and}\quad (\tilclass{\A}_{{\rm qc}}, dg(\B)\cap \Ch(\Flat(X))).$$
To see that e.g.\ the former one is indeed a cotorsion pair, we have to check that $\Ch(\A_{{\rm qc}}) = {}^\perp(\Ch(\A_{{\rm qc}})^{\perp}\cap \Ch(\Flat(X))) \cap \Ch(\Flat(X))$. The inclusion ``$\subseteq$'' is clear. To see the other one, pick $\mathscr X_\bullet \in {}^\perp(\Ch(\A_{{\rm qc}})^{\perp}\cap \Ch(\Flat(X))) \cap \Ch(\Flat(X))$ and consider a short exact sequence $0 \to \mathscr B_\bullet \to \mathscr A_\bullet \to \mathscr X_\bullet \to 0$ with $\mathscr A_\bullet \in \Ch(\A_{{\rm qc}})$ and $\mathscr B_\bullet \in \Ch(\A_{{\rm qc}})^{\perp}$. As $\A_{{\rm qc}} \subseteq \Flat(X)$ and $\Ch(\Flat(X))$ is a resolving class, we infer that $\mathscr B_\bullet \in \Ch(\Flat(X))$. Thus the sequence splits and $\mathscr X_\bullet$ is a direct summand of $\mathscr A_\bullet$, hence an element of $\Ch(\A_{{\rm qc}})$. The proof for the latter cotorsion pair goes in a similar way.

Now we will apply \cite[Lemma 4.3]{G} to these two complete cotorsion pairs in the category $\Ch(\Flat(X))$ and to the thick class $\W=\tilclass{\Flat}(X)$ in $\Ch(\Flat(X))$. So we need to check that the following conditions hold:
\begin{enumerate}
\item[(i)] $\tilclass{\A}_{{\rm qc}}=\Ch(\A_{{\rm qc}})\cap \tilclass{\Flat}(X)$.
\item[(ii)] $\Ch(\A_{{\rm qc}})^{\perp}\cap \Ch(\Flat(X))\subseteq \tilclass{\Flat}(X)$.
\end{enumerate}
Since every flat $\A_R$-periodic module is trivial and the classes $\A$ and $\Flat$ are Zariski-local, we immediately infer that every flat $\A_{{\rm qc}}$-periodic quasi-coherent sheaf is trivial. Thus, from Proposition \ref{prop.periodic}, we get condition (i). So let us see condition (ii). Let $\mathscr L_\bullet\in\Ch(\A_{{\rm qc}})^{\perp}\cap \Ch(\Flat(X)) $. Since the pair $(\Ch(\Flat(X)),\Ch(\Flat(X))^\perp)$ in $\Ch(\Qcoh(X))$ has enough injectives, there exists an exact sequence, $$0\to \mathscr L_\bullet\to \mathscr P_\bullet\to \mathscr M_\bullet\to 0, $$ with $\mathscr P_\bullet\in \Ch(\Flat(X))^\perp$ and $\mathscr M_\bullet\in \Ch(\Flat(X))$. Now, since $\mathscr L_\bullet\in\Ch(\Flat(X))$, we get that $\mathscr P_\bullet\in \Ch(\Flat(X))\cap \Ch(\Flat(X))^\perp=\tilclass{\Flat\Cot}(X)$. By Lemma \ref{l.keylemma}, there exists an exact sequence $$0\to \mathscr K
_\bullet\to \mathscr F_\bullet\to \mathscr M_\bullet\to 0,$$ where $\mathscr F_\bullet\in \Ch(\A_{{\rm qc}})$ and $\mathscr K_\bullet\in\tilclass\Flat(X)$. Now, we take the pull-back of $\mathscr P_\bullet\to \mathscr M_\bullet$ and $\mathscr F_\bullet \to \mathscr M_\bullet$, so we get a commutative diagram:
 \begin{equation*}
    \xymatrix{
   & & 0\ar[d] & 0\ar[d] & & \\
      &  & \mathscr K_\bullet \ar[d] \ar@{=}[r]  &\mathscr K_\bullet \ar[d]\\
      0\ar[r] & \mathscr L_\bullet\ar@{=}[d]\ar[r] & \mathscr Q_\bullet \ar[r] \ar[d]\ar[r] &\mathscr F_\bullet \ar[r]\ar[d] &0\\
      0\ar[r] & \mathscr L_\bullet\ar[r] & \mathscr P_\bullet \ar[r]\ar[d]  &\mathscr M_\bullet\ar[r]\ar[d] &0 \\
              &               &   0            &           0    }
  \end{equation*}
In the middle column, the complexes $\mathscr K_\bullet$ and $\mathscr P_\bullet$ belong to $\tilclass{\Flat}(X)$. Therefore, the complex $\mathscr Q_\bullet$ also belongs to $\tilclass{\Flat}(X)$. Since $\mathscr F_\bullet\in \Ch(\A_{{\rm qc}})$ and $\mathscr L_{\bullet}\in  \Ch(\A_{{\rm qc}})^{\perp}$, the exact sequence in the middle row splits. So, $\mathscr L_\bullet \in \tilclass{\Flat}(X)$ as desired.

Therefore by \cite[Lemma 4.3]{G} we have the exact model structure in $\Ch(\Flat(X))$ given by the triple 

 $$\M_{\A_{{\rm qc}}}^{\Ch_{\rm flat}}=(\Ch(\A_{{\rm qc}}),\tilclass{\Flat}(X), dg(\B)\cap \Ch(\Flat(X))).$$ Since it has the same class of trivial objects, this model is Quillen equivalent to the flat model in $\Ch(\Flat(X))$, $$\M_{\rm flat}^{\Ch_{\rm flat}}=(\Ch(\Flat(X)),\tilclass{\Flat}(X),dg(\Cot(X))\cap \Ch(\Flat(X)) ).$$ 
This is, in turn, the restricted model of the model  $$\M_{\rm flat}=(\Ch(\Flat(X),\W, dg(\Cot(X)))$$ in $\Ch(\Qcoh(X))$ with respect to the exact category $\Ch(\Flat(X))$ of cofibrant objects. Thus, $\M_{\rm flat}$ and $\M_{\rm flat}^{\Ch_{\rm flat}}$ are canonically Quillen equivalent. 
To finish the proof, let us show that the model $\M_{\A_{{\rm qc}}}^{\Ch_{\rm flat}}$ is Quillen equivalent to 
$$\M_{\A_{{\rm qc}}}=(\Ch(\A_{{\rm qc}(X)}),\W_1, dg(\B)).$$ But this model is canonically Quillen equivalent to its restriction to the cofibrant objects, i.e. $$\M_{\A_{{\rm qc}}}^{\Ch_{\A_{\rm qc}}}=(\Ch(\A_{{\rm qc}}),\tilclass{\A}_{{\rm qc}}, dg(\B)\cap \Ch(\A_{{\rm qc}})).$$ Finally the Quillen equivalent cofibrant restricted model of $$\M_{\A_{{\rm qc}}}^{\Ch_{\rm flat}}=(\Ch(\A_{{\rm qc}}),\tilclass{\Flat}(X), dg(\B)\cap \Ch(\Flat(X)))$$ is given by the triple $$(\Ch(\A_{{\rm qc}}),\tilclass{\Flat}(X)\cap\Ch(\A_{{\rm qc}}), dg(\B)\cap \Ch(\A_{{\rm qc}}))), $$ which by condition (i) above is precisely the previous model $\M_{\A_{{\rm qc}}}^{\Ch_{\A_{\rm qc}}}$. In summary, we have the following chain of Quillen equivalences among the several models, $$ \M_{\rm flat}\simeq \M_{\rm flat}^{\Ch_{\rm flat}}\simeq  \M_{\A_{{\rm qc}}}^{\Ch_{\rm flat}}\simeq \M_{\A_{{\rm qc}}}^{\Ch_{\A_{\rm qc}}}\simeq \M_{\A_{{\rm qc}}}. $$ The first and the last models give our desired Quillen equivalence.
\end{proof}
Recall from \cite{D} that $\mathscr M\in \Qcoh(X)$ is an \emph{infinite-dimensional vector bundle} if, for each $U\in \mathcal U$, the $\mathscr O_X(U)$-module $\mathscr M(U)$ is projective. We will denote by $\Vect(X)$ the class of all infinite-dimensional vector bundles on $X$. In case $\Vect(X)$ contains a generating set of $\Qcoh(X)$, we know from \cite[Corollary 3.15 and 3.16]{EGPT} that the pair $(\Vect(X),\mathcal B)$ (where $\mathcal B:=\Vect(X)^{\perp}$) is a complete cotorsion pair generated by a set. It is hereditary, because the class $\Vect(X)$ is resolving. Thus by \cite[Theorem 4.10]{G} we get the abelian model structure $\M_{\textrm{vect}}=(\Ch(\Vect(X)),{\W_1}, dg(\B))$ in $\Ch(\Qcoh(X)) $ given by the two complete hereditary cotorsion pairs:
$$(\Ch(\Vect(X)),\Ch(\Vect(X))^{\perp})\quad \text{and}\quad (\tilclass{\Vect}(X), dg(\B)).$$ We will denote by ${\bf D}(\Vect(X))$ its homotopy category.

We are now in position to prove Corollary 2 in the Introduction.
\begin{cor}\label{eq.vbundles}
Let $X$ be a scheme with enough infinite-dimensional vector bundles. Then the categories ${\bf D}(\mathrm{Flat}(X))$ and $\mathbf D\mathbb(\Vect(X))$ are triangle equivalent, the equivalence being induced by a Quillen equivalence between the corresponding model categories.
\end{cor}
\begin{proof}
The proof will follow by showing that  ${\bf D}(\mathrm{Vect}(X))$ and $\mathbf D\mathbb(\VF(X))$ are Quillen equivalent, and then by applying Corollary \ref{cor.triang.equiv.flatveryflat}. To this end, we will prove that the model structures $\M_{\mathrm{vect}}$ and $\M_{\VF}$ have the same trivial objects. To achieve this, by \cite[Theorem 1.2]{G2}, it suffices to show that the trivial fibrant and cofibrant objects of one structure are trivial also in the other structure. This assertion is clearly satisfied by the trivial cofibrants of $\M_{\mathrm{vect}}$ and trivial fibrants of $\M_{\VF}$, as
\[ \widetilde{\Vect}(X) \subseteq \widetilde{\VF}(X) \quad \text{and} \quad \Ch(\VF(X))^\perp \subseteq \Ch(\Vect(X))^\perp.\]

Now let $\shcplx V \in \widetilde{\VF}(X)$; since there are enough infinite-dimensional vector bundles, the cotorsion pair $(\tilclass{\Vect}(X), dg(\B))$ has enough projectives, hence there is a short exact sequence
\[ 0 \to \shcplx Q \to \shcplx P \to \shcplx V \to 0 \]
with $\shcplx P \in \tilclass{\Vect}(X)$. Restricting this to an open affine subset of $X$, we obtain a short exact sequence with a complex of projective modules in the middle and ending in a complex of very flat modules, and the objects of cycles also belonging to the respective classes. Since the projective dimension of very flat modules does not exceed $1$, it follows that $\shcplx Q$ has also projective cycles after this restriction, hence $\shcplx Q \in \tilclass{\Vect}(X)$. We conclude that $\shcplx V$, being a factor of two trivial objects, is itself trivial in $\M_{\mathrm{vect}}$.

Finally, pick $\shcplx M \in \Ch(\Vect(X))^\perp$. Using the completeness of the cotorsion pair $(\Ch(\VF(X)), \Ch(\VF(X))^\perp)$, we obtain a short exact sequence
\[ 0 \to \shcplx K \to \shcplx V \to \shcplx M \to 0 \]
with $\shcplx V \in \Ch(\VF(X))$ and $\shcplx K \in \Ch(\VF(X))^\perp$. As $\shcplx K$ is trivial in $\M_{\mathrm{vect}}$, it suffices to show that $\shcplx V$ is trivial, too. Furthermore, $\Ch(\VF(X))^\perp \subseteq \Ch(\Vect(X))^\perp$ implies that in fact, $\shcplx V \in \Ch(\Vect(X))^\perp$. So as above, construct a short exact sequence
\[ 0 \to \shcplx Q \to \shcplx P \to \shcplx V \to 0, \]
this time with $\shcplx P \in \Ch(\Vect(X))$ and $\shcplx Q \in \Ch(\Vect(X))^\perp$. The same local argument as above shows that $\shcplx Q \in \Ch(\Vect(X))$, and we also have $\shcplx P \in \Ch(\Vect(X))^\perp$ (being an extension of two objects from the class). Hence $\shcplx V$ is a factor of two complexes from the class $\Ch(\Vect(X)) \cap \Ch(\Vect(X))^\perp$, which is a subclass of $\tilclass{\Vect}(X)$ and consequently $\tilclass{\VF}(X)$, therefore consisting of trivial objects of $\M_{\VF}$.
\end{proof}

Finally, the last consequence is also an application of Theorem \ref{t.mc.general case} for the class of very flat quasi-coherent sheaves. It follows from Gillespie \cite[Theorem 4.10]{G}
\begin{cor}
There is a recollement

$$\xymatrix{
\mathbf D_{\mathrm{ac}}(\VF(X)) \ar[rr]|j
&& \mathbf D(\VF(X)) \ar@/_1pc/[ll]\ar@/^1pc/[ll]\ar[rr]|w &&\mathbf D(X) \ar@/_1pc/[ll]\ar@/^1pc/[ll] }$$

\end{cor}

\medskip\par\noindent

\begin{remark}
Murfet and Salarian deal in \cite{MS} with a suitable generalization of total acyclicity for sche\-mes. Namely, they define the category ${\bf D}_{{\rm F}\textrm{-}{\rm tac}}(\Flat(X))$\footnote{\,The terminology used in \cite{MS} is ${\bf D}_{{\rm tac}}(\Flat(X))$.}  of \emph{F-totally acyclic complexes} in ${\bf D}(\Flat(X))$ and prove that, in case $X=\Spec(R)$ is affine and $R$ is Noetherian of finite Krull dimension, ${\bf D}_{{{\rm F}\textrm{-}{\rm tac}}}(\Flat(X))$ is triangle equivalent to ${\bf K}_{\rm tac}({\Proj}(R))$ (the homotopy category of totally acyclic complexes of projective modules) showing that ${\bf D}_{{\rm F}\textrm{-}{\rm tac}}(\Flat(X))$ also constitutes a good replacement of ${\bf K}_{\rm tac}({\Proj}(R))$ in a non-affine context. An analogous version of Theorem \ref{t.mc.general case} allows to restrict the equivalence between ${\bf D}(\mathrm{Flat}(X))$ and $\mathbf D\mathbb(\A_{{\rm qc}})$ to their corresponding categories of F-totally acyclic complexes  ${\bf D}_{{\rm F}\textrm{-}{\rm tac}}(\A_{\rm qc})$ and ${\bf D}_{{\rm F}\textrm{-}{\rm tac}}(\Flat(X))$. In particular, the full subcategory ${\bf D}_{{\rm F}\textrm{-}{\rm tac}}(\VF(X))$ of F-totally acyclic complexes of very flat quasi-coherent sheaves in ${\bf D}(\VF(X))$ is triangle equivalent with Murfet's and Salarian's derived category of F-totally acyclic complexes of flats.

\end{remark}

\section*{acknowledgements}
We would like to thank Jan \v S\v t'ov\' \i\v cek for many useful comment and discussions during the preparation of this manuscript. We would also like to thank Leonid Positselski for his inspiring work \cite{P} and for sharing his knowledge on the subject of very flat modules.

\end{document}